\documentclass[11pt]{amsart} 

\input xypic

\usepackage{verbatim}
\usepackage{rotating}
\usepackage{amssymb}
\usepackage{amsmath}

\numberwithin{equation}{section}

\newtheorem*{proposition1.2}{Proposition 1.2}
\newtheorem*{proposition2.2}{Proposition 2.2}
\newtheorem*{theorem2.3}{Theorem 2.3}
\newtheorem*{theorem3.7}{Theorem 3.7}
\newtheorem*{theorem3.9}{Theorem 3.9}
\newtheorem*{theorem5.4}{Theorem 5.4}
\newtheorem*{theorem5.7}{Theorem 5.7}
\newtheorem*{corollary6.4}{Corollary 6.4}
\newtheorem{Theorem}{Theorem}[section]
\newtheorem{proposition}[Theorem]{Proposition}
\newtheorem{corollary}[Theorem]{Corollary}
\newtheorem{lemma}[Theorem]{Lemma}

\theoremstyle{definition}
\newtheorem*{definition5.6}{Definition 5.6}
\newtheorem{definition}[Theorem]{Definition}
\newtheorem{remark}[Theorem]{Remark}

\newtheorem{conjecture}[Theorem]{Conjecture}

\let\tilde=\widetilde
\let\hat=\widehat

\newcommand{\cS}{\mathcal{S}}
\newcommand{\cM}{\mathcal{M}}
\newcommand{\cP}{\mathcal{P}}
\newcommand{\F}{\mathbb{F}}

\title{Universal Abelian H-spaces}
\author{Brayton Gray}
\address{Department of Mathematics, Statistics and Computer Science\\
         University of Illinois at Chicago\\
         851 S.~Morgan Street\\
         Chicago, IL, 60607-7045, USA} 
\email{brayton@uic.edu}

\begin{document}


\maketitle

By an Abelian H-space we mean a connected CW 
complex with an H-space structure that is homotopy
associative and homotopy commutative. A space $T$ provided
with a map $i\colon X\to T$ will be called the Abelianization of~$X$
if $T$ is an Abelian H-space which satisfies the following
universal property: for any Abelian H-space $Z$ and any map
$f\!\colon X\to Z$ there is an H-map $\hat{f}\!\colon T\to Z$, unique up to
homotopy
such that $\hat{f}i\sim f$. Given a space~$X$, an Abelianization may or
may not exist, but if it does it is unique. Thus an
Abelianization plays a role for Abelian H-spaces analogous
to the James construction for group like spaces.

The problem of constructing an Abelianization for suitable
spaces has received much recent attention. (\cite{G4}, \cite{AG}, \cite{G5}, \cite{Gr2}, \cite{Gr1}, \cite{N3}, \cite{T1}, \cite{T2}). This
concept was first discussed
in~\cite{G4}, which took up the question of whether the Anick space
$T_{2n-1}(p^r)$
is the Abelianization of the Moore space $P^{2n}(p^r)=S^{2n-1}\cup
{}_{p^r}e^{2n}$. The Anick spaces (see \cite{AG}, \cite{A}) are certain H-spaces that lie
in a fibration sequence:
\[
\Omega^2S^{2n+1}\xrightarrow{\pi_n} S^{2n-1}\xrightarrow{}
T_{2n-1}(p^r)\xrightarrow{}\Omega S^{2n+1}
\]
where $\pi_n$ is the Cohen--Moore--Neisendorfer map of degree~$p^r$.
One of our conclusions is that $T_{2n-1}(p^r)$ is not the
Abelianization of~$P^{2n}(p^r)$. This conclusion is not consistent
with the results of~\cite{T1}, where several arguments appear
to have gaps.

In~\cite{G4}, a secondary EHP sequence was developed
involving the spaces $T_{2n-1}(p^r)$ and spaces $T_{2n}(p^r)$ ($T_{2n}(p^r)$ is the
fiber of the degree $p^r$ map on~$S^{2n+1}$). Some sort of universal
property was needed to form the compositions necessary for
the EHP machine.  $T_{2n}(p^r)$ was shown to be the
Abelianization of the Moore space $P^{2n+1}(p^r)$, but the properties of $T_{2n-1}(p^r)$ were unresolved.

The goal of this work is to study the conditions under
which an Abelianization exists as well as to seek a more
restricted universal property which is appropriate for the 
Anick spaces.

In particular, we prove the following:
\begin{proposition1.2}
Suppose $T$ is an Abelianization of $X$ and $k=\F_p$ or $Q$. Then
\[
H_*(T;k)\cong \cS (H_*(X;k))
\]
where $\cS(V)$ is the symmetric algebra on the vector space~$V$.
\end{proposition1.2}

We will say that a graded Abelian group $G$ has torsion
exponent~$p^r$ if each element of~$G$ of finite order has
order~$\leq p^r$.
\begin{proposition2.2}
Suppose $T$ is the Abelianization of~$X$. Then
the torsion exponent of~$H^*(T)$ is the same as the torsion
exponent of~$H^*(X)$.
\end{proposition2.2}

The next result severely impacts the question of which
spaces~$X$ allow an Abelianization. 
\begin{theorem2.3}
Suppose that $H_*(X)$ has bounded torsion and
that $p^{\text{th}}$ powers are trivial in $H^*(X;Z/p)$. Then if an
Abelianization $T$ exists for~$X$, $H_{2n-1}(X)$ is free for all~$n$.
\end{theorem2.3}

Every known example of an Abelianization $T$ of a
space~$X$ satisfies the condition that the $H$-map $\Omega
SX\xrightarrow{h}T$
extending the inclusion $i\colon X\to T$ has a right homotopy
inverse. We call such an Abelianization split. In
section~\ref{sec3} we study H-spaces $T$ together with a map
$h\colon\Omega SX\to T$ which is split. In this case we construct a
fibration sequence:
\begin{equation*}
\tag{3.1} \Omega SX\xrightarrow{h} T\xrightarrow{} R\xrightarrow{\pi}SX
\end{equation*}
where $R$ is a co-H space. We also introduce the
``universal Whitehead product'' map (See~\cite{T1})
$W=\nabla\omega$
\[
\Omega SX\ast \Omega SX\xrightarrow{\omega}SX\vee SX\xrightarrow{\nabla}SX
\]
where $\omega$ is the fiber of the inclusion: $SX\vee SX\to SX\times SX$
and $\nabla$ is the folding map. We prove:
\begin{theorem3.7}
Suppose $T$ is Abelian and
split. Then if $\pi$ factors through~$W$, $T$ is the Abelianization
of~$X$
\[
\xymatrix{
&\Omega SX\ast \Omega SX\ar@{->}[d]_{W}\\
R\ar@{-->}[ur]
\ar@{->}[r]^{\pi}&SX.
}
\]
\end{theorem3.7}

In section~\ref{sec4} we survey the known examples
of Abelianization and analyze them from the point of
view of section~\ref{sec3}.

In section~\ref{sec5} we return to the problem of finding a suitable
universal property for the Anick spaces. We will call an H-space.
$Z$ a suitable target for $P^{2n}(p^r)$ iff each map $f\!\colon
P^{2n}(p^r)\to Z$ has a unique extension to an H-map $\hat{f}\!\colon T_{2n-1}(p^r)\to Z$. We
introduce the following torsion condition:
\begin{definition5.6}
A space $Y$ is $(n,r)$-subexponential if for each
$s\geqslant 1$,
\[
p^{r+s-1}\pi _{np^s}(Y;Z/p^{r+s})=0.
\]
\end{definition5.6}
We then conclude that a generalized Eilenberg--MacLane
space is a suitable target for~$T_{2n-1}(p^r)$ iff it is $(2n,r)$
subexponential. In fact:
\begin{theorem5.7}
\textup{(a)}\hspace*{0.5em}If $Z$ is any space which is $(2n,r)$-subexponential,
then every map $f\!\colon P^{2n}(p^r)\to \Omega Z$ extends to a
map $\hat{f}\!\colon T_{2n-1}(p^r)\to Z$.

\textup{(b)}\hspace*{0.5em}If $Z$ is a $(2n,r)$-subexponential H-space and there is an
H-map
\[
\hat{f}\!\colon T_{2n-1}(p^r)\to Z
\]
extending $f\!\colon P^{2n}(p^r)\to
Z$, then $\hat{f}$ is unique
up to homotopy.
\end{theorem5.7}

The question of whether the Anick spaces are Abelian
H-spaces is as yet unresolved. In~\cite{T1} an argument
for the affirmative was presented, but it relied on some of
the same unpublished details used in the claim that  the
Anick spaces are the Abelianization of the Moore space. In
section~\ref{sec6} we prove:
\begin{corollary6.4}
If $T_{2n-1}(p^r)$ is homotopy associative and~$Z$ is a
$(2n,r)$-subexponential double loop space, then~$Z$
is a suitable target for~$T_{2n-1}(p^r)$.
\end{corollary6.4}

Throughout this paper, all spaces will be connected,
localized at a fixed prime~$p$, and of finite type. I
would like to thank Jim Lin for help with section~\ref{sec2},
Fred Cohen, and especially Joe Neisendorfer
for many helpful discussions.

\section{}\label{sec1}

In this section we will assume given a space~$X$
with an Abelianization~$T$. Given a graded vector space~$V$, we
will write $\cS(V)$ for the symmetric algebra on~$V$. This is the
free commutative algebra generated by $V$ and is a tensor
product of polynomial algebras on a basis for~$V$ modulo
the relation $2x^2=0$ when $x$ has odd degree. If $V$ is a
co-algebra, $\cS(V)$ is a Hopf algebra.

For later purposes we will only assume that the
map $i\colon X\to T$ is universal for a selected list of target
spaces.
\begin{proposition}\label{prop1.1}
Suppose $k=\F_p$ or $Q$, and $i\colon X\to T$ is
universal for the targets $\Omega^2S^2X$ and $K(k,m)$ for all $m$.
Then
\[
H_*(T;k)\cong \cS(H_*(X;k)).
\]
\end{proposition}
\begin{proof}
Since $T$ is Abelian, there is a map of Hopf
algebras:
\[
\theta\colon \cS(H_*(X;k))\to H_*(T;k).
\]
We first show that $\theta$ is a monomorphism. Using the
universal property we construct an H-map:
\[
\gamma\colon T\to \Omega^2S^2X
\]
extending the inclusion of~$X$ in $\Omega^2S^2X$. We claim that
the composition:
\[
\cS(H_*(X;k))\xrightarrow{\theta}H_*(T;k)\xrightarrow{}H_*(\Omega^2S^2X;k)
\]
is a monomorphism. In the case that $k=\F_{p}$, this follows
from the work of Cohen~\cite[3.2]{C} where the author shows
that $H_*(\Omega^2S^2X;\F_p)$ is a symmetric algebra on
certain generators which include the generators of
$H_*(X;k)$. Thus $\theta$ is a monomorphism in this case. In case
$k=Q$, choose a basis $\{x_i\}$ for $H_*(X;Q)$ and let
$\hat{x}_i$ be a dual basis. Then the map:
\[
\Pi \hat{x}_i\colon X\to \Pi K(Q,|x_i|)
\]
is a monomorphism in rational homology. Using
the universal property for each factor, we construct an
extension:
\[
\xymatrix{
X\ar@{->}[rr]^{\Pi\hat{x}_i\makebox[15pt]{}}\ar@{->}[dr]_{i}&&\Pi K(Q,|x_i|)\\
&T\ar@{->}[ur]_{e}&.
}
\]
Since $H_*(\Pi K(Q,|x_i|))$ is the symmetric algebra on the
$x_i$, it follows that $\theta$ is a monomorphism in this case
as well.

We will show that $\theta$ is an epimorphism by
showing that the dual
\[
\theta^*\colon H^*(T;k)\to \cS^*(H_*(X;k))
\]
is a monomorphism. Choose an element $\xi \in H^n(T;k)$
in the kernel of~$\theta^*$ of least possible dimension.

Let $\mu\colon T\times T\to T$ be the H-space structure map, so
\[
\mu^*\colon H^*(T;k)\to H^*(T\times T;k)\cong H^*(T;k)\otimes H^*(T;k)
\]
defines the coalgebra structure. Then
\[
(\theta^*\otimes \theta^*)(\mu^*(\xi))=\nabla^*(\theta^*(\xi))=0.
\]
Suppose $\mu^*(\xi)=\xi\otimes
1\sum\xi_i^{\prime}\otimes\xi_i^{\prime\prime}+1\otimes\xi$ where
$\xi_i^{\prime}$ and $\xi_i^{\prime\prime}$
have dimension less than~$n$. We then get:
\[
(\theta^*\otimes\theta^*)(\mu^*(\xi))=\sum\theta^*(\xi_i^{\prime})\otimes\theta^*(\xi_i^{\prime\prime}).
\]
However $\theta^*\otimes\theta^*$ factors:
\[
H^*(T;k)\otimes H^*(T;k)\xrightarrow{\theta^*\otimes 1}\cS^*\otimes
H^*(T;k)\xrightarrow{1\otimes \theta^*}\cS^*\otimes \cS^*.
\]
Since $\theta^*$ is a monomorphism in dimensions less than~$n$
this composition
is a monomorphism when restricted to $H^i(T;k)\otimes H^j(T;k)$
with $i,j<n$. It follows that $\xi$ is primitive. Represent
$\xi$ on H-map
\[
\xi\colon T\to K(k,n).
\]

We need only show that $\xi i\colon X\to K(k,n)$
is trivial to conclude that $\xi$ is trivial and hence $\theta^*$ is a 
monomorphism. But $i^*$ factors:
\[
H^*(T;k)\xrightarrow{\theta^*}\cS^*\to H^*(X;k)
\]
since the inclusion $H_*(X;k)\to H_*(T;k)$ factors through
$\cS(H_*(X;k))$. So $i^*(\xi)=0$ since $\theta^*(\xi)=0$.
\end{proof}
Consequently a much weaker universal property---for example
restricting acceptable targets $Z$ to be
double loop spaces---is sufficient to calculate the
homology. As a corollary, we have
\begin{proposition}\label{prop1.2}
Suppose $T$ is the Abelianization of~$X$
and $k=\F_p$ or $Q$. Then
\[
H_*(T;k)\cong \cS(H_*(X;k)).
\]
\end{proposition}
\section{}\label{sec2}
In this section we will develop some general
properties of an Abelianization~$T$ of~$X$ and conclude
that there are some sever limitations on the spaces~$X$
that are available.

Recall that if $Z$ is an H space a class $\xi\in H^m(Z;G)$
is called primitive if the corresponding map:
\[
\xi\colon Z\to K(G,m)
\]
is an H-map. This is equivalent to the equation:
\[
\mu^*(\xi)=\pi_1^*(\xi)+\pi_2^*(\xi)
\]
where $\mu\colon Z\times Z\to Z$ is the H-space structure map,
and $\pi_i$ is the $i^{\text{th}}$ projection. Let $PH^*(Z;G)$ be
the subgroup of primitive cohomology classes.
\begin{lemma}\label{lem2.1}
Suppose $T$ is the Abelianization of~$X$.
Then
\[
i^*\colon PH^*(T;G)\to H^*(X;G)
\]
is an isomorphism.
\end{lemma}
\begin{proof}
This is immediate from the universal property.
\end{proof}

We will say that a space has cohomology torsion
exponent $p^r$ if each cohomology class of finite order
has order dividing $p^r$. Since all spaces we are considering
are of finite type, the cohomology torsion exponent is the
same as the homology torsion exponent.
\begin{proposition}\label{prop2.2}
Suppose $X$ has cohomology exponent $p^r$ and~$T$
is the Abelianization of~$X$. Then $T$ has cohomology exponent~$p^r$.
\end{proposition}
\begin{proof}
Since the inclusion $X\to \Omega^2X^2X$ extends over~$T$, $S^2X$
is a retract of~$S^2T$, so the exponent of~$X$ is less than the
exponent of~$T$. Suppose $X$ has cohomology exponent~$p^r$
and $\xi\in H^m(T)$ has order $p^{r+1}$. Suppose that $m$ is the
minimal dimension in which such a class~$\xi$ can occur. We will
show that $p^r\xi$ is primitive. This will imply that
$p^r\xi=0$ by~\ref{lem2.1}, completing the proof. To accomplish this
we examine the cohomology group $H^m(T\wedge T)$ using a K\"{u}nneth
Theorem available for all spaces of finite type~\cite[5.7.26]{HW}
\[
\tilde{H}^m(T\wedge T)\cong\bigoplus_{i+j=m}\tilde{H}^i(T)\otimes
\tilde{H}^j(T)\oplus\bigoplus_{i+j=m+1}\text{Tor}(\tilde{H}^i(T),\tilde{H}^j(T)).
\]
Since $\tilde{H}^{1}(T)$ is free, all nonzero terms on the right
involve cohomology
groups of dimension less than~$m$; hence $H^m(T\wedge T)$
has homology exponent at most~$p^r$. Thus $\mu^*(p^r\xi)$ has no
middle terms and $p^r\xi$ is primitive, completing the proof.
\end{proof}

The next result is particularly useful when either $X$ is a co-H
space or all Steenrod operations are trivial in~$X$.
\begin{Theorem}\label{theor2.3}
Suppose $H_*(X)$ has bounded torsion and $p^{\text{th}}$
powers are trivial in $H^*(X;Z/p)$. Then if $X$ has an Abelianization,
$H_{2n-1}(X)$ is free for all~$n$.
\end{Theorem}
\begin{proof}
Suppose $T$ is an Abelianization of~$X$. We first
show that all $p^{\text{th}}$ powers are trivial in $H^*(T;Z/p)$. The map
\[
\cP\colon H^*(T;Z/p)\to H^*(T;Z/p)
\]
defined by $\cP(\xi)=\xi^p$ is a map of Hopf algebras, as is its
dual
\[
\cP_*\colon H_*(T;Z/p)\to H_*(T;Z/p).
\]
$\cP_*$ is determined by its action on generators, all of which are in the
image of $H_*(X;Z/p)$. Since $\cP$ is trivial when applied to~$X$,
$\cP_*$ is as well. Thus $\cP_*$ and $\cP$ are trivial when applied to~$T$.

We will suppose now that there is an element $\tilde{y}\in H_{2n-1}(X)$
of finite order with $r>1$ and $r$ maximal. Using the exact
sequence:\addtocounter{equation}{3}
\begin{equation}\label{eq2.4}
\dots\to
H_{k+1}(X;Z/p)\xrightarrow{\partial}H_k(X)\xrightarrow{p}H_k(X)\xrightarrow{\rho}H_k(X;Z/p)\to\dots
\end{equation}
we choose $x^{\prime}\in H_{2n}(X;Z/p)$ with $\partial
x^{\prime}=p^{r-1}\tilde{y}$ and let $y^{\prime}=
\rho(\tilde{y})$. Then $\beta^{(r)}(x^{\prime})=y^{\prime}$. Choose dual
classes $\hat{x}^{\prime}$ and $\hat{y}^{\prime}$
with $\langle x^{\prime},\hat{x}^{\prime}\rangle=1=\langle
y^{\prime},\hat{y}^{\prime}\rangle$ and
$\beta^{(r)}\hat{y}^{\prime}=\hat{x}^{\prime}$. Now extend
$\hat{x}^{\prime}$ and $\hat{y}^{\prime}$ to primitive cohomology classes
$\hat{x}$ and $\hat{y}$ in $H^*(T;Z/p)$.
Let $x=i_*(x^{\prime})$ and $y=i_*(y^{\prime})$ so $\langle
x,\hat{x}\rangle=1=\langle y,\hat{y}\rangle$. We now
appeal to the theorem of Browder:\addtocounter{Theorem}{1}
\begin{Theorem}[{\cite[5.4]{Br}}]\label{theor2.5}
Let $T$ be an  Abelian H-space and
$x\in E_{2n}^{(r)}$, $\beta^{(r)}(x)=y\in E_{2n-1}^{(r)}$ where $E^{(*)}_*$
is the Backstein
spectral sequence. Then if $r>1$ $\beta^{(r+1)}(x^p)=\{x^{p-1}y\}$---the
class of $x^{p-1}y$ in~$E_{2n-1}^{(r+1)}$.
\end{Theorem}

Note that Browder's theorem does not assert that
$\{x^{p-1}y\}\neq 0$. Our first task will be to determine this.
Suppose that $x^{p-1}y=\beta^{(s)}(z)$ for some $z$ and $s\leq r$.
Observe, by induction on~$k$, that
\[
\langle x^ky,\hat{x}^k\hat{y}\rangle=k!
\]
since $\hat{x}$ and $\hat{y}$ are primitive. Thus
\[
(p-1)!=\langle x^{p-1}y,\hat{x}^{p-1}\hat{y}\rangle=\langle
\beta^{(s)}(z),\hat{x}^{p-1}\hat{y}\rangle =\langle
z,\beta^{(s)}(\hat{x}^{p-1}\hat{y})\rangle.
\]
But $\beta^{(s)}\hat{y}=0$ and $\beta^{(s)}\hat{x}=0$ for $s<r$, so we
must have $s=r$. Since $\beta^{(r)}\hat{y}=\hat{x}$, we get
\[
(p-1)!=\langle z,\hat{x}^p\rangle.
\]
Since $p^{\text{th}}$ powers are trivial in $H^*(T;Z/p)$, we conclude
that no such $z$ can exist and $\beta^{(r+1)}(x^p)\neq 0$ in
$E^{(r+1)}_{2np-1}$.
It follows that there is an element of order $p^{r+1}$ in
$H_*(T)$. By~\ref{prop2.2} there must be an element of $H_*(X)$ of order~$p^r$
contradicting the choice of~$r$ being maximal. This completes the proof in case $r>1$.

Now suppose that each element of $H_{2n-1}(X)$ of finite order
has order~$p$. Write:
\[
H_*(X)=F\oplus T_e\oplus T_o
\]
where $F$ is free with a basis $\{\tilde{z}_1,\dots,\tilde{z}_m\}$, $T_e$
is a sum
of cyclic groups of even degree generated by
$\{\tilde{e}_1,\dots,\tilde{e}_{\ell}\}$ and $T_o$
is a sum of cyclic groups of odd degree generated by
$\{\tilde{y}_1\dots\tilde{y}_k\}$.
Using \ref{eq2.4}, construct corresponding classes in $H_*(X;Z/p)$
$z_i^{\prime}$, $e_2^{\prime}$, $y_i^{\prime}$, $f_i^{\prime}$ and
$x_i^{\prime}$ with $\beta_{(s_i)}(f_i^{\prime})=e_i^{\prime}$ and
$\beta^{(1)}(x_i^{\prime})=y_i^{\prime}$. Using a monomial basis, construct
dual
classes 
$\hat{z}_i^{\prime}$, $\hat{e}_i^{\prime}$, $\hat{y}_i^{\prime}$,
$\hat{f}_i^{\prime}$,
$\hat{x}_i^{\prime}$
in $H^*(X;Z/p)$
and extend these to
primitive cohomology classes 
$\hat{z}_i$, $\hat{e}_i$, $\hat{y}_i$,
$\hat{f}_i$,
$\hat{x}_i$
in $H^*(T;Z/p)$.
Write
${z}_i$, ${e}_i$, ${y}_i$,
${f}_i$ and
${x}_i$
for the images of
${z}_i^{\prime}$, ${e}_i^{\prime}$, ${y}_i^{\prime}$,
${f}_i^{\prime}$,
${x}_i^{\prime}$
in $H_*(T;Z/p)$.

Suppose now that $2n=|x_1|\leq |x_i|$ for all~$i$, and we relabel
$x_1$ as~$x$.  Since $\beta^{(s)}(x^p)=0$, $\partial x^p$ is divisible by~$p$.
Since $|\partial x^p|$
is odd, this implies that $\partial x^p=0$ so $x^p=\rho(\overline{x})$ for
some
class $\overline{x}\in H_{2np}(T)$. We claim that $\overline{x}$ has
infinite order.
If not $x^p=\beta^{(s)}(\theta)$ for some $\theta\in E^{(s)}_{2np+1}$. Let
$J\subset H_*(T;Z/p)$
be the subalgebra generated by $e_i$, $f_i$ $z_i$. Then for $\xi \in J$,
$\beta^{(s)}(\xi)\in J$ if it is defined. Now
\[
E^{(2)}=Z/p\left[x_1^p,\dots,x_k^p\right]\otimes\Lambda\left(y_1x_1^{p-1},\dots,
y_kx_k^{p-1}\right)\otimes E^{(2)}(J).
\]
For $s\geqslant 2$, $E^{(s)}$ consists only of classes of elements listed
here.
For dimensional reasons, any possible $\theta$ must be of the form
\[
\theta=\sum \sigma_i x_i^p+\sigma_i^{\prime}y_ix_i^{p-1}+\sigma
\]
where $\sigma_i$, $\sigma_i^{\prime}$ and $\sigma$ are elements of~$J$.
Furthermore, $|x_i| =2n$ for
each $i$ that can occur in this expression.
Now $|\sigma_i|=1$ and $|\sigma_i^{\prime}|=2$. Suppose then that
$\beta^{(i)}(\theta)=0$
for all $i<s$; then we have
\[
\beta^{(s)}(\theta)=\sum
\sigma_i^{\prime}\beta^{(s)}\left(y_ix_i^{p-1}\right)+\beta^{(s)}(\sigma)
\]
which is in the ideal generated by~$J$. Since $x^p$ is not in this
ideal, we conclude that $\beta^{(s)}(\theta)\neq x^p$ for any $s$ or $\theta$
and $\overline{x}$ has
infinite order.

Now let $H_*(T)=\overline{F}\oplus\overline{T}$ where $\overline{F}$ is
free and $\overline{T}$ is a
torsion group. We claim that $\overline{F}=\cS_Z(F)$, the symmetric
ring over $Z$ of~$F$. Consider the commutative
diagram:
\[
\xymatrix@R=2pt{
&\cS_Z(F)\ar@{->}[r]^{\alpha}&H_*(T)\\
&\ar@{{>}->}[dddd]&\ar@{->}[dddd]\\
&&\\
&&\\
&&\\
\cS_Z(F)\otimes Q\simeq\kern-26pt& \cS_Q(F)\ar@{->}[r]_{\simeq\makebox[10pt]{}}&H_*(T;Q)
.}
\]
Since the left and bottom homomorphisms are monomorphisms,
$\alpha$ is too and $\cS_Z(F)\subset \overline{F}$. However the rank of
$\cS_Z(F)$ is
the same as the rank of $\cS_Q(F)\simeq H_*(T;Q)=\overline{F}\otimes Q$, so
$\cS_Z(F)$ and $\overline{F}$ have the same rank. Consequently
$\overline{F}=\cS_Z(F)$
and $\overline{x}\in \overline{F}$ must be of the form:
\[
\overline{x}=f(\overline{z}_1,\dots,\overline{z}_m)
\]
for some polynomial~$f\!$. Apply $\rho$ to this equation to get
\[
x_1^p=\theta(z_0,\dots, z_m)
\]
which is a contradiction.
\end{proof}

Note: We have few examples of spaces which allow an
Abelianization, and \ref{prop2.2} and \ref{theor2.3} suggest that the
possibilities are limited. We suggest here one
method of finding Abelianization. This comes from the
observation that if~$X$ and $X\cup e^n$ both have an
Abelianization, there is a fibration
\[
T(X)\to T(X\cup_{\theta}e^n)\to T(S^n).
\]
This is particularly useful if $n$ is odd and $p>3$. In this case
$T(S^n)=S^n$ and $T(X\cup e^*)$ can be constructed
from clutching data:
\[
S^{n-1}\times T(X)\to T(X)
\]
obtained from the map $\theta\colon S^{2n-1}\to X$ and the H-space
structure on~$T(X)$.

\section{}\label{sec3}
In this section we will begin with an H-space $T$
and a map $i\colon X\to T$. We will not assume that ~$T$ is the
Abelianization of~$X$. Recall the Hopf
quasi fibration (\cite{DL}, \cite{S})
\[
T\to T\ast T\to ST
\]
where the inclusion $T\to T\ast T$ is null homotopic. We
construct an induced fibration sequence:
\begin{equation}\label{eq3.1}
\begin{split}
\xymatrix{
\Omega ST\ar@{->}[r]^{\mu}&T\ar@{->}[r]\ar@{=}[d]&T\ast T
\ar@{->}[r]&ST\\
\Omega SX\ar@{->}[u]_{\Omega si}\ar@{->}[r]^{h}&T
\ar@{->}[r]&R
\ar@{->}[r]^{\pi}
\ar@{->}[u]_{j}&SX
\ar@{->}[u]_{Si}.
}
\end{split}
\end{equation}
Note that if $T$ is homotopy associative, $\mu$ is an H-map (\cite[Proposition~A3]{G7})
and consequently $h$ is an H-map as well.\addtocounter{Theorem}{1}
\begin{definition}\label{def3.2}
We will say that $T$ is split if the map
$h\colon \Omega SX\to T$ has a right homotopy inverse.
\end{definition}
\begin{proposition}\label{prop3.3}
Suppose $T$ and $SX$ are atomic.\footnote{Being atomic implies that the
first nontrivial homotopy group is cyclic, and any self map inducing an isomorphism in this cyclic group is a homotopy equivalence.} Then
the following are equivalent:

\textup{(a)}\hspace*{0.5em}$T$ is split

\textup{(b)}\hspace*{0.5em}the inclusion $Si\colon SX\to ST$ has a left homotopy
inverse

\textup{(c)}\hspace*{0.5em}Given a map$f\!\colon X\to Z$ where $Z$ is an H-space,
there is an extension $\hat{f}\!\colon T\to Z$; i.e., $\hat{f} i\sim f$.
\end{proposition} 
Note: We are not asserting that $\hat{f}$ is an H-map in~\textup{(c)} or that
it is unique.
\begin{proof}
\textup{(a)}${}\Rightarrow{}$\textup{(b)}.\hspace*{0.5em}Let $g\colon T\to \Omega SX$ be a
right homotopy inverse
for~$h$. Suppose that $X$ is $(n-1)$-connected and $\pi_n(X)\neq 0$.
Then $g$ and $h$ induce inverse isomorphisms in
$\pi_n=H_n$. Thus the adjoint of $g$
\[
\tilde{g}\colon ST\to SX
\]
induces an isomorphism of $H_{n+1}$. Consequently
the composition:
\[
SX\xrightarrow{Si} ST\xrightarrow{\tilde{g}}SX
\]
induces an isomorphism in $H_{n+1}$. Since $SX$ is atomic,
$e=\tilde{g}(Si)$ is an equivalence. Then $e^{-1}\tilde{g}$ is a
left homotopy inverse to~$Si$.

\textup{(b)}${}\Rightarrow{}$\textup{(c)}.\hspace*{0.5em}Let $\tilde{g}\colon ST\to SX$ be a
left homotopy inverse to~$Si$.
Consider the composition:
\[
T\xrightarrow{g}\Omega SX\xrightarrow{f_{\infty}}Z
\]
where $g$ is the adjoint of~$\tilde{g}$ and $f_{\infty}$ is an extension
of~$f$ given
by a choice of association in the James construction. Then
$\tilde{f}i\sim f$.

\textup{(c)}${}\Rightarrow{}$\textup{(a)}.\hspace*{0.5em}Using \textup{(c)}, construct a map $g\colon
T\to \Omega SX$ such that $gi$
is the inclusion of~$X$ in $\Omega SX$. Taking adjoints, we see that
\[
SX\xrightarrow{Si}ST\xrightarrow{\tilde{g}}SX
\]
is homotopic to the identity. So~\textup{(b)} follows and
$hgi\sim i$. Suppose that $T$ is $n-1$ connected and $\pi_n(T)$
is a nontrivial cyclic group. Then $i_{*}\colon\pi_n(X)\to \pi_n(T)$
is an isomorphism since a cyclic group has no
nontrivial retracts. It follows that
\[
(hg)_*\colon \pi_n(T)\to \pi_n(T)
\]
is an isomorphism as well. Since $T$ is atomic, $hg$ is an
equivalence and $g(hg)^{-1}$ is a splitting.
\end{proof}
\begin{remark}\label{rem3.4}
It is easy to see that if $T$ is atomic, $X$ is
atomic; for any map $f\!\colon X\to X$ induces an H-map $\hat{f}\colon T\to
T$.\
If $X$ is $n-1$ connected and $f$ induces an isomorphism  in $\pi_n(X)$,
then $\hat{f}$ will induce an isomorphism in $\pi_n(T)$. Hence $\hat{f}$ is
an H-equivalence.
Since $PH^*(T;Z/p)\simeq H^*(X;Z/p)$, it follows that $f$
induces isomorphisms in cohomology and hence it is an
equivalence.
\end{remark}
\begin{proposition}\label{prop3.5}
If $T$ is split, $R\vee ST\simeq SX\rtimes T$ and hence
$R$ is a co-H space.
\end{proposition}
\begin{proof}
Each fibration:
\[
F\to E\to SK
\]
is given by a clutching construction (\cite[1b]{G3}) and there is
a natural homotopy equivalence:
\[
E/F\simeq SK \rtimes F.
\]
Applying this we see that
\[
R/T\simeq SX\rtimes T
\]
Since $T$ is split, the map $T\to R$ is null homotopic,
so $R/T\simeq R\vee ST$.
\end{proof}
\begin{proposition}\label{prop3.6}
Suppose $T$ is split and $Z$ is an H-space.
Then two H-maps $f_0,f_1\colon T\to Z$ are homotopic iff $f_0i\sim
f_1i\colon X\to Z$.
\end{proposition}
\begin{proof}
Let $J(X)$ be the James construction on~$X$ filtered by
$J_k(X)$. Let $h_k\colon J_k(X)\to T$ be the composition:
\[
J_k(X)\to J(X)\to \Omega SX\xrightarrow{h}T.
\]
Let $\overline{\mu}\colon T\times T\to T$ be the H-space
structure map. According to (\cite[A3]{G5}), $h_k$ satisfies the
inductive formula:
\[
h_k(x_k,x_{k-1},\dots x_1)=\overline{\mu}(i(x_k),h_{k-1}(x_{k-1},\dots,
x_1)).
\]
Thus we have a homotopy commutative diagram:
\begin{equation*}
\xymatrix@C=40pt{
X\times J_{k-1}(X)\ar@{->}[r]^{\makebox[14pt]{}i\times h_{k-1}}\ar@{->}[d]_{a}&T\times
T\ar@{->}[r]^{f_i\times f_i}\ar@{->}[d]^{\overline{\mu}}&Z\times Z
\ar@{->}[d]
\\
J_k(X)\ar@{->}[r]^{h_k}&T
\ar@{->}[r]^{f_i}&Z
}
\end{equation*}
where $i=0,1$. Suppose $f_0i\sim f_1i$ and $f_0h_{k-1}\sim f_1h_{k-1}$.
Then
$f_0h_ka\sim f_1h_ka$. Since $S(a)$ has a right homotopy
inverse and~$Z$ is an H-space, we conclude that $f_0h_k\sim f_1h_k$.
However these homotopies may not be compatible. This is 
similar to a phantom map situation where the filtration is
not by skeleta or compact subsets but by the James
filtration. This can be handled by the same methods (see \cite{G1}). Define
\[
\bigvee=\bigvee^{\infty}_{k=0}J_k(X)
\]
and construct a cofibration sequence:
\[
V \to J(X)\to C(X)\to SV\to SJ(X).
\]
The map $SV\to SJ(X)$ has a right homotopy inverse, so the
map $SJ(X)\to SC(X)$ is null homotopic. Now the difference
\[
S(f_0h)-S(f_1h)\colon SJ(X)\to SZ
\]
is null homotopic when restricted to $SV$ since $f_0h_k\sim f_1h_k$ for
all~$k$. Thus it factors over $SC(X)$ and is consequently
inessential. Thus $S(f_0h)\sim S(f_1h)$. Since $Z$ is an H-space
$f_0h\sim f_1h\colon \Omega SX\to T\to Z$.
Since $T$ is split, we have $f_0\sim f_1$.
\end{proof}

We see that if $T$ is split extensions exist and if
we have an extension which is
an H-map, it is unique. The hard part is to construct
extensions which are H-maps. The next result addresses
this issue.

Recall (\cite{T1}), the universal Whitehead product $W$, which
is the composition:
\[
\Omega SX\ast\Omega SX\xrightarrow{\omega}SX\vee SX\xrightarrow{\nabla}SX
\]
where $\omega$ is the fiber of the inclusion of $SX\vee SX$ in $SX\times SX$
and $\nabla$ is the folding map.
\begin{Theorem}\label{theor3.7}
Suppose $T$ is Abelian and
split. Then if $\pi$ factors through~$W$, $T$ is the Abelianization
of~$X$
\[
\xymatrix{
&\Omega SX\ast \Omega SX\ar@{->}[d]_{W}\\
R\ar@{-->}[ur]
\ar@{->}[r]^{\pi}&SX.
}
\]
\end{Theorem}
\begin{proof}
Suppose $Z$ is Abelian and $f\!\colon X\to Z$. Extend $f$ to an
H-map $f_{\infty}\!\colon \Omega SX\to Z$. Now consider the diagram:\addtocounter{equation}{6}
\begin{equation}\label{eq3.8}
\begin{split}
\xymatrix{
\Omega(SX\vee SX)\ar@{->}[r]^{\makebox[10pt]{}\Omega\nabla\kern-10pt}\ar@{->}[d]^{\Omega j}
&\Omega SX\ar@{->}[r]^{f_{\infty}}&Z\\
\Omega (SX\times SX)\ar@{}[r]\kern0pt\cong\kern-40pt&\Omega SX\times \Omega SX
\ar@{->}[r]^{\makebox[20pt]{}f_{\infty}\times f_{\infty}}&Z\times Z\ar@{->}[u].
}
\end{split}
\end{equation}
Since the equivalence $\Omega(SX\times SX)\cong \Omega SX\times \Omega SX$
is an
H-equivalence, all maps in this diagram are H-maps. Thus
to check that it is homotopy commutative, it suffices to
check the restriction to $X\vee X$, where both sides are  $f\nabla$.
From this we conclude that $(f_{\infty})\Omega W\sim \ast$, since $\Omega
W=(\Omega\nabla)(\Omega \omega)$
and $(\Omega j)(\Omega \omega)\sim \ast$. The hypothesis that $\pi$ factors
though
$W$ implies that $f_{\infty}\Omega \pi\sim\ast$.

Now let $g\colon T\to \Omega SX$ be a right homotopy inverse
to~$h$. Since $T$ is homotopy associative, $h$~is an H-map,
so $h(1-gh)\sim\ast$. This implies that there is a map
$s\colon \Omega SX\to \Omega R$ such that $(\Omega\pi)s\sim 1-gh$.
Consequently
\[
f_{\infty}-f_{\infty}gh\sim f_{\infty}(1-gh)\sim f_{\infty}(\Omega \pi)
s\sim \ast.
\]
So $f_{\infty}\sim f_{\infty}gh$. Let $\hat{f}=f_{\infty}g\colon T\to Z$.
Let $i^{\prime}\colon X\to \Omega SX$
be the inclusion. Then $\hat{f}i=f_{\infty}gi\sim f_{\infty}ghi^{\prime}\sim
f_{\infty}i^{\prime}\sim f$
so $\hat{f}$ is an extension of~$f$. To see that $\hat{f}$ is an H-map,
consider the diagram:
\[
\xymatrix{
\Omega SX\times \Omega SX\ar@{->}[r]^{\makebox[16pt]{}h\times h}\ar@{->}[d]&T\times
T\ar@{->}[r]^{\hat{f}\times\hat{f}}\ar@{->}[d]&Z\times Z\ar@{->}[d]\\
\Omega SX\ar@{->}[r]^{h}&T\ar@{->}[r]^{\hat{f}}&Z
}
\]
where the verticle maps are the H-space structure maps. The
left hand square and the rectangle commute up to homotopy.
Since $h\times h$ has a right homotopy inverse, the right hand
square does as well so $\hat{f}$ is an H-map.
 To see that $\hat{f}$ is unique,
suppose ${\hat{f}}^{\,\prime}\!\colon T\to Z$ is another H-map extending $f$.
Then ${\hat{f}}^{\,\prime}h\colon \Omega SX\to Z$ is an H-map excluding $f$,
so ${\hat{f}}^{\,\prime}h\sim \hat{f}h$. Since $h$ has a right homotopy inverse,
${\hat{f}}^{\,\prime}\sim\hat{f}$.
\end{proof}
\addtocounter{Theorem}{1}
\begin{corollary}\label{cor3.9}
Suppose $T$ is Abelian and $\pi$ factors through $W$. Then there is a
homotopy commutative diagram:
\[
\xymatrix{
R\ar@{->}[r]^{\alpha\makebox[21pt]{}}\ar@{->}[dr]_{\pi}
&\Omega SX\ast \Omega SX\ar@{->}[d]^W\ar@{->}[r]^{\makebox[21pt]{}\beta}
&R\ar@{->}[dl]^{\pi}\\
&SX&
}
\]
with $\beta\alpha$ a homotopy equivalence.
\end{corollary}
\begin{proof}
By~\ref{theor3.7}, $\pi$ factors through~$W$. To see that $W$ factors
through $\pi$, consider the diagram~(\ref{eq3.8}) with $T$ in place of~$Z$
and $h$ in place of $f_{\infty}$. It follows that $h(\Omega W)\sim*$.
So $\Omega W$ factors through $\Omega\pi$. Since $\Omega SX*\Omega SX$ is a co-H space,
$W$ factors through $\pi$. Thus we have
constructed $\alpha$ and $\beta$. It remains to show that $\beta\alpha$ is an
equivalence. Looping the diagram, we get:
\[
\xymatrix{
\Omega R\ar@{->}[r]^{\Omega \alpha\makebox[21pt]{}}\ar@{->}[dr]_{\Omega \pi}
&\Omega(\Omega SX\ast \Omega SX)\ar@{->}[d]^{\Omega W}\ar@{->}[r]^{\makebox[21pt]{}\Omega\beta}
&\Omega R\ar@{->}[dl]^{\Omega \pi}\\
&\Omega SX&
}
\]
Since $\Omega \pi$ induces a monomorphism in$\!\!\mod p$
homology $(\Omega\beta)(\Omega\alpha)$ does as well. Since $\Omega R$ is of finite
type, this implies that $(\Omega\beta)(\Omega\alpha)$ induces an isomorphism in$\!\!\mod p$ homology.
Likewise $(\Omega\beta)(\Omega \alpha)$ induces an isomorphism
in rational homology, so it is an equivalence and $\beta\alpha$ is an
equivalence.
\end{proof}
\begin{remark}\label{rem3.10}
When $T$ is the Abelianization of $X$, the algebraic structure can be
described as follows: Let $L$ be the free Lie algebra on $H_*(X;k)$ and $U(L)$
its universal enveloping algebra. Then $H_*(\Omega SX;k)\cong U(L)$ and
$H_*(T;k)=U(L/[L,L])$. $[L,L]\subset L$ is the free Lie algebra on the
desuspension of $\widetilde{H}_*(R;k)$ and $H_*(\Omega R;k)=U([L,L])$.
\end{remark}
\begin{remark}\label{rem3.11}
Since $H_*(\Omega SX;k)\to H_*(T;k)$ is the standard
map from a tensor algebra to its Abelianization in the case
that~$T$ is the Abelianization of~$X$, $H_*(\Omega R;k)$ is the universal
enveloping algebra $U([L,L])$ where $L$ is the Lie algebra
given by $H_*(\Omega SX;k)$. Thus the image of the map
\[
H_*(R;k)\to H_*(\Omega SX;k)
\]
is $[L,L]$ where $L$ is $H_*(\Omega SX;k)$ considered as a Lie algebra $L$.
\end{remark}

\section{}\label{sec4}
In this section we will briefly review the known examples of
spaces~$X$ with an Abelianization~$T$. These occur in various
places in the literature. They are all
split and we will examine them in terms of the fibration
sequence~\ref{eq3.1}.

\textup{(a)}\hspace*{0.5em}$X=S^{2n}$, $T=\Omega S^{2n+1}$, $R=*$. $T$ is Abelian
iff $p>2$.

\textup{(b)}\hspace*{0.5em}$X=S^{2n+1}=T$, $R=S^{4n+3}$, $\pi=[i,i]\colon
S^{4n+3}\to S^{2n+2}$. $T$~is Abelian
iff $p>3$.

\textup{(c)}\hspace*{0.5em}$X=P^{2n+1}(p^r)$, $T=T_{2n}(p^r)=S^{2n+1}\{p^r\}$---the
fiber of the degree $p^r$ map
$p^r\colon S^{2n+1}\to S^{2n+1}$. The structure map is given by:
\[
\pi\colon R=\bigvee_{k\geqslant 1}S^{2nk+2}X\xrightarrow{\bigvee
ad^{k-1}(u)([v,v])} P^{2n+2}(p^r)
\]
where $v\colon P^{2n+1}(p^r)\to P^{2n+1}(p^r)$ is the identity map and
\[
u=\beta^r v\colon P^{2n}(p^r) \to P^{2n+1}(p^r).
\]
This splitting first appeared in \cite[1.1]{CMN}.
$T$ is Abelian if $p>3$ (\cite{N1}). (See also~\cite{G5}).

\textup{(d)}\hspace*{0.5em}$X=S^{2m}\cup_{\theta}e^{2n+1}$, $T={}$the fiber of
$S\theta\colon S^{2n+1}\to S^{2m+1}$. This is an
immediate generalization of~\textup{(c)}, however the proof of the splitting
is quite different (\cite{G5}) and simpler. The structure
map is given by:
\[
\pi\colon R=\bigvee_{k\geqslant 0}S^{2mk+2n+2}X\xrightarrow{\bigvee
ad^{(k)}(u)([v,v])}SX
\]
where $u$ and $v$ are defined similarly to~\textup{(c)} and the corresponding
Whitehead
products with coefficients in~$X$ are in the sense of~\cite{G6}. The
proof that $T$ is Abelian for $p>3$ is due to Grbic~(\cite{Gr1}).

\textup{(e)}\hspace*{0.5em}$X=S^{2n-1}\cup_{\omega_n}e^{2np-2}$---$2np-2$ skeleton
of~$\Omega^2S^{2n+1}$. $T=\Omega J_{p-1}(S^{2n})$,
where $J_{p-1}(S^{2n})$ is the $(p-1)$st filtration of the James
construction $J(S^{2n})$
\[
J_{p-1}(S^{2n})=S^{2n}\cup e^{4n}\cup\dots\cup e^{2n(p-1)}
\]
\[
R=\bigvee_{k\geqslant 0}S^{(2np-2)k+4n-1}\vee\bigvee_{k\geqslant
0}S^{(2np-2)(k+1)+2n}\xrightarrow{\pi}SX
\]
is given by $ad^{(k)}(v)([u,u])\vee ad^{(k+1)}(v)(u)$. Here by $[u,u]$ we
mean the composition:
\[
S^{4n-1}\xrightarrow{[i,i]}S^{2n}\xrightarrow{\ }SX
\]
and by $[v,\phi]$, for $\phi\colon S^m\to SX$, we mean the unique map
(up to homotopy in the homotopy commutative square:
\[
\xymatrix{
S^{2np-2+m}
\ar@{->}[r]^{[v,\phi]}
&SX\\
\ar@{->}[u]
\ar@{->}[r]^{\omega}
S^mX
\ar@{->}[u]
&S^m\vee SX.
\ar@{->}[u]_{\phi\vee 1}
}
\]
This result appears in a number of places (\cite{G4}, \cite{G5}, \cite{H}).

\textup{(f)}\hspace*{0.5em}$X=S^{2n+1}\cup_{\omega_n}e^{2np-2}\cup_{pi}e^{2n-1}$---the
$2np-1$ skeleton of $\Omega^2S^{2n+1}$.
In this case $T$ is the Selick space $F_2$ (\cite{Se}) given by the
pull back:
\[
\xymatrix{
F_2
\ar@{->}[r]
\ar@{->}[d]
&\Omega^2S^{2n+1}
\ar@{->}[d]^{H}
\\
S^{2np-1}
\ar@{->}[r]
&\Omega^2S^{2np+1}.
}
\]
This is the most complicated example and has been thoroughly
worked out by Grbic~\cite{Gr2}.

\textup{(g)}\hspace*{0.5em}We might try to generalize this to other spaces in the
nested sequence:
\[
\Omega J_{p-1}(S^{2n})\subset F_2\subset \Omega J_{p^2-1}(S^{2n})\subset
F_3\subset\dots.
\]
However none of these spaces beyond~$F_2$ is the Abelianization
of some space~$X$. This is because in each case there is an
indecomposable homology class
$x_2\in H_{2np^2-2}(T;Z/p)$ such that $(\cP^1)_*(x_2)\in
H_{(2np-2)p}(T;Z/p)$ is
decomposable. Thus $x_2$ must lie in the image of $i_*\colon H_*(X;Z/p)
\to H_*(T;Z/p)$ modulo decomposables, while $(\cP^1)_*(x_2)$ does
not lie in the image of~$i_*$.

\textup{(h)}\hspace*{0.5em}Let $X$ be a $p$-local $CW$ complex with $\ell$ cells,
each of
which has odd dimension such that $\ell <p-2$. Then~$X$ has an
Abelianization (\cite{T2}). These spaces were first discussed
in~\cite{CHZ}.

\section{}\label{sec5}
The origin of the study of Abelianization arose from an
attempt to understand the mapping properties of the Anick
spaces $T_{2n-1}(p^r)$ (\cite{G4}). Recall that the space $T_{2n}(p^r)$
is
the Abelianization of the Moore space $P^{2n+1}(p^r)$. It is clear from
\ref{prop2.2}, \ref{theor2.3}, or~\ref{theor5.2} (below) that $T_{2n-1}(p^r)$ is not the
Abelianization
of any subspace. Nevertheless, the considerations of~\cite{G4} suggest
that it should have some universal mapping property.
\begin{definition}\label{def5.1}
Let $M(n,r)$ be the category of H-spaces $Z$ such
that there is a 1--1 correspondence between $[P^{2n}(p^r),Z]$ and the
set of homotopy classes of H-maps from $T_{2n-1}(p^r)$ to~$Z$.
\end{definition}

We seek to understand this category. Examining
the proofs of~\ref{prop2.2} and~\ref{theor2.3} suggests that we first
inquire as to
which Eilenberg--MacLane spaces belong to~$M(n,r)$. We
consider $K(G,k)$ where~$G$ is a finitely generated Abelian
group. The homotopy classes of H-maps from $T_{2n-1}(p^r)$
to $K(G,k)$ are in \mbox{1--1} correspondence with $PH^k(T_{2n-1}(p^r);G)$.
It suffices to consider the cases $G=Z_{(p)}$ and $G=Z/p^r$.
\begin{Theorem}\label{theor5.2}
The primitive cohomology classes in $T_{2n-1}(p^r)$ are
given by:

\textup{(a)}\hspace*{0.5em}$PH^k(T_{2n-1}(p^r))=\begin{cases}
Z/p&\text{if $k=2np^s$ $s\geqslant0$},\\0&\text{otherwise}.
\end{cases}$

\textup{(b)}\hspace*{0.5em}$PH^k(T_{2n-1}(p^r);Z/p^t)=\begin{cases}
Z/p&\text{if $k=2np^s$ and $t\geqslant r+s$},\\
Z/p&\text{if $k=2np^s-1$ and $t\geqslant r+s$},\\
0&\text{otherwise}.
\end{cases}$
\end{Theorem}
\begin{proof}
\textup{(a)}\hspace*{0.5em}$H^k(T_{2n-1}(p^r))\neq 0$ iff $k=2ni$ and the projection
map:
\[
T_{2n-1}(p^r)\to \Omega S^{2n+1}
\]
is onto in cohomology. Let $e_k\in H^{2nk}(\Omega S^{2n+1})$ be the dual of
the
$k^{\text{th}}$ power of some chosen (and fixed) generator in
$H_{2n}(\Omega S^{2n+1})$. Then the diagonal map is given by:
\[
\mu^*(e_k)=\sum_{i+j=k}e_i\otimes e_j.
\]
Let $v_k$ be the image of $e_k$ in $H^{2nk}(T_{2n-1}(p^r))$. The order of
$v_k$ is
$p^{\nu(k)+r}$ where $\nu(k)$ is the number of powers of~$p$ in~$k$.
Suppose now that $\alpha_k$ is an integer such that $\alpha_kv_k$ is
primitive. If $k=p^s\ell$ with $(p,\ell)=1$ and $\ell>1$, $\mu^*(v_k)$
contains the term $v_{p^s}\otimes v_{p^s(\ell-1)}$ which has the same order
as~$v_k$. So $\alpha_kv_k$ cannot be a nonzero primitive in this
case. However the middle terms of $\mu^*(v_{p^s})$ all have
orders~$p^{r+v(i)}$ where $0<i<p^s$. These orders all divide $p^{r+s-1}$ and include $p^{r+s-1}$  when $i=p^{s-1}$. Thus $p^{r+s-1}v_s$ generates the primitives and
has order~$p$.

\textup{(b)}\hspace*{0.5em}From the exact sequence:
\[
0\to
H^{2nk-1}(T_{2n-1};Z/p^t)\xrightarrow{\delta}H^{2nk}(T_{2n-1})\xrightarrow{p^t}H^{2nk}(T_{2n-1})
\]
we see that if $\theta\in H^{2nk-1}(T_{2n-1};Z/p^t)$ is primitive,
$\delta(\theta)$ is
as well, so if $\theta$ is nonzero, $\delta(\theta)=p^{r+s-1}v_{p^s}$ up 
to a
unit. Consequently $k=p^s$. If $t\geqslant r+s$, $v_{p^s}$ is also in the
image of $\delta$, so $\theta$ is divisible by $p^{r+s-1}$.
Suppose $\theta=p^{r+s-1}\gamma$. Then the image of~$\theta$
under the homomorphism:
\[
\mu^*-\pi_1^*-\pi_2^*\colon H^{2np^s-1}(T_{2n-1};Z/p^t)\to
H^{2np^s-1}(T_{2n-1}\times T_{2n-1};Z/p^t)
\]
must be zero since
$\mu^*-\pi_1^*-\pi_2^*$ factors through $H^{2np^s-1}(T_{2n-1}\wedge
T_{2n-1};Z/p^t)$
and all the elements of this group have order
dividing $p^{r+s-1}$. So in case $t\geqslant r+s$, $\theta$ is primitive.
Suppose on the other hand that $d=r+s-t>0$. We will
show that there are no primitives in $H^{2np^s-1}(T_{2n-1};Z/p^t)$.
To see this note that $Z/p^t$ is an injective $Z/p^t$ module.
Consequently
\[
H^m(W;Z/p^t)\cong \text{Hom}(H_m(W;Z/p^t),Z/p^t).
\]
It follows that there is a nonzero primitive class~$\theta$ in
$H^{2np^s-1}(T_{2n-1};Z/p^t)$ iff there is a class in
$H_{2np^s-1}(T_{2n-1};Z/p^t)$
which is not in the image of the homorphism:
\[
\mu_*-\pi_{1*}-\pi_{2*}\colon H_{2np^s-1}(T_{2n-1}\times T_{2n-1};Z/p^t)\to
H_{2np^s-1}(T_{2n-1};Z/p^t)
\]

A calculation with the homology Serres spectral
sequence for the fibration:
\[
S^{2n-1}\to T_{2n-1}\to \Omega S^{2n+1}
\]
with $Z/p^t$ coefficients shows that there is a
subalgebra
\[
Z/p^t\left[v^{\prime};2np^{t-r}\right]\otimes\bigwedge\left(u^{\prime};
2np^{t-r}-1\right)
\]
of $H_*(T_{2n-1};Z/p^t)$. Since $H_{2np^s-1}(T;Z/p^t)$ has only
one generator, it must be $(v^{\prime})^{p^{d}-1}u^{\prime}$ up to a unit.
This
is in the image of $\mu_*-\pi_{1*}-\pi_{2*}$ iff $d>0$, so
$PH^{2np^s-1}(T_{2n-1};Z/p^t)=0$ when $d>0$.
\end{proof}
\begin{corollary}\label{cor5.3}
Suppose $E$ is a generalized Eilenberg--MacLane space.
Then $E\in M(n,r)$ iff $\forall s>0$ we have:

\textup{(a)}\hspace*{0.5em}$p^{r+s-1}\pi_{2np^s}(E)=0$

\textup{(b)}\hspace*{0.5em}the torsion subgroup of $\pi_{2np^s-1}(E)$ has exponent at
most $p^{r+s-1}$.
\end{corollary}
\begin{proof}
Since there are H-maps:
\[
K(G_k,k)\xrightarrow{\alpha_k}E\xrightarrow{\beta_k}K(G_k,k)
\]
with $\beta_k\alpha_k\sim 1$, $E\in \cM(n,r)$ iff $K(G_k,k)\in \cM(n,r)$
for each~$k$. This is trivial if $k\leqslant 2n$. Thus for $k>2n$ this is
equivalent to the statement that there are no nontrivial
H-maps from $T_{2n-1}(p^r)$ to $K(G_k,k)$. It suffices to check this
in the cases $G_k=Z_{(p)}$ and $G_k=Z/p^t$. We apply \ref{theor5.2}. There
is a nontrivial H-map from $T_{2n-1}(p^r)$ to $K(Z_{(p)},k)$ iff $k=2np^s$
with $s\geqslant 1$ so
this condition is equivalent to the statement that
$\pi_{2np^s}(E)$ is a torsion group  for each $s\geqslant 1$. There are no
nontrivial H-maps from $T_{2n-1}(p^r)$ to $K(Z/p^t,k)$ iff $t<r+s$
whenever $k=2np^s$ or $2np^s-1$ with~$s\geqslant 1$.
\end{proof}

The problem with targets that are
Eilenberg--MacLane spaces is consequently not one
of existence but of uniqueness. The next result shows that
these conditions affect uniqueness in a more general context.
\begin{Theorem}\label{theor5.4}
Suppose $\hat{f}_1,\hat{f}_2\!\colon T_{2n-1}(p^r)\to Z$ are two H-maps
extending $f\!\colon P^{2n}(p^r)\to Z$, where $Z$ is an
Abelian H-space. Suppose also that for each $s\geqslant 1$:
\begin{enumerate}
\item[\textup{(a)}]
$p^{r+s-1}\pi_{2np^s}(Z)=0$

\item[\textup{(b)}]
the torsion subgroup of $\pi_{2np^s-1}(Z)$ has exponent at
most $p^{r+s-1}$
\end{enumerate}
then $f_1\sim f_2$.
\end{Theorem}
\begin{proof}
Let $g=\hat{f}_1-\hat{f}_2$. Since $Z$ is Abelian, $g$ is an H-map.
Let $Z^{[m]}$ be the $m^{\text{th}}$ Postnikov section of~$Z$ and let
\[
\pi_m\colon Z\to Z^{[m]}
\]
be the projection. We will show by induction on $n$ that $\pi_mg\sim *$.
This is clearly true when $m\leqslant 2n$. Consider the fibration
sequence:
\[
K(\pi_{m+1}(Z),m+1)\xrightarrow{\partial}Z^{[m+1]}\xrightarrow{\
}Z^{[m]}\xrightarrow{k_m}K(\pi_{m+1}(Z),m+2).
\]
Assuming that $\pi_mg\sim *$ we can factor $\pi_{m+1}g$ through
$K(\pi_{m+1}(Z),m+1)$. Choose any map $\gamma\colon T_{2n-1}(p^r)\to
K(\pi_{m+1}(Z),m+1)$
with $\partial \gamma\sim\pi_mg$. We claim that $\gamma$ is an H-map. Let
\[
\delta(\gamma)\colon T_{2n-1}(p^r)\wedge T_{2n-1}(p^r)\to
K(\pi_{m+1}(Z),m+1)
\]
be the H-deviation. 
Since $\delta(\gamma)$ is an H-map and $\partial$ is an
H-map
$\delta(\gamma)$ factors through $\Omega k_m$:
\[
T_{2n-1}(p^r)\wedge T_{2n-1}(p^r)\xrightarrow{\epsilon} Omega Z^{[m]}\xrightarrow{\Omega
k_m}K(\pi_{m+1}(Z),m+1).
\]
$\epsilon$ is adjoint to the composition:
\[
ST_{2n-1}(p^r)\wedge T_{2n-1}(p^r)\to
Z^{[m]}\xrightarrow{k_m}K(\pi_{m+2}(Z),m+2).
\]
By \cite[4.5]{GT}, $ST_{2n-1}(p^r)\wedge T_{2n-1}(p^r)$ is a one point
union of Moore spaces up to homotopy. However for each
Moore space $P^{\ell}$ and map $\varphi\colon P^{\ell}\to Z^{[m]}$ we
have $k_m\varphi\sim *$; this
follows since any nontrivial map $\varphi\colon P^{\ell}\to Z^{[m]}$ could
only
occur when $\ell \leqslant m+1$. consequently $\gamma$ is an H-map.
Applying
\ref{cor5.3} we see that $\gamma\sim *$ so $\pi_{m+1}g\sim *$. Since
$H^k(T_{2n-1}(p^r))$ is finite for each
$k$, there are no phantom maps (See \cite{G1} or~\cite{G2}). Thus $g\sim
*$.
\end{proof}

The conditions \textup{(a)} and \textup{(b)} in \ref{cor5.3} and \ref{theor5.4} will clearly
play a
role in the eventual understanding of the mapping properties
of~$T_{2n-1}$. We note the following.
\begin{lemma}\label{lem5.5}
Suppose $Y$ is of finite types and $b>0$. Then\linebreak[4]
$p^a\pi_m(Y;Z/a+ b)=0$ iff $p^a\pi_m(Y)=0$ and the torsion
in $\pi_{m-1}(Y)$ has order at most $p^a$.
\end{lemma}
\begin{proof}
Consider the ladder of long exact sequences:
\[
\xymatrix{
\pi_m(Y)
\ar@{}[dr]|{(A)}
\ar@{->}[r]^{p^{a+b}}
\ar@{->}[d]_{p^a}&
\pi_m(Y)
\ar@{}[dr]|{(B)}
\ar@{->}[r]
\ar@{->}[d]_{p^a}&
\pi_m(Y;Z/p^{a+b})
\ar@{->}[r]
\ar@{->}[d]_{p^a}
&\pi_{m-1}(Y)
\ar@{->}[r]^{p^{a+b}}
\ar@{->}[d]_{p^a}&
\pi_{m-1}(Y)
\ar@{->}[d]_{p^a}\\
\pi_m(Y)
\ar@{->}[r]^{p^{a+b}}&
\pi_m(Y)
\ar@{->}[r]
&\pi_m(Y;Z/p^{a+b})
\ar@{->}[r]&
\pi_{m-1}(Y)
\ar@{->}[r]^{p^{a+b}}&
\pi_{m-1}(Y).
}
\]
Assume first that $p^a\pi_m(Y;Z/p^{a+b})=0$. Then from
square $\textup{(B)}$ we see that $p^a\pi_m(Y)\subset p^{a+b}\pi_m(Y)$. Since
$b>0$, this implies that the elements of $p^a\pi_m(Y)$ are
divisible by arbitrarily high powers of~$p$. Since $Y$ is of
finite types, we conclude that $p^a\pi_m(Y)=0$. Suppose
there is an element $\xi\in \pi_{m-1}(Y)$ of order $p^n$
where $n>a$. Then
\[
p^{a+b}\left(p^{n-a-1}\xi\right)=p^{n+b-1}\xi=0
\]
so $p^{n-a-1}\xi$ is in the image of $\pi_m(Y;Z/p^{a+b})$. Since
$p^a\pi_m(Y;Z/p^{a+b})=0$, it follows that $p^{n-1}\xi=p^a(p^{n-a-1}\xi)=0$.

Conversely suppose that $p^a\pi_m(Y)=0$ and
all torsion in $\pi_{m-1}(Y)$\linebreak[4]
 has order dividing $p^a$. Let
$\xi\colon P^m(p^{a+b})\to Y$ represent an element of\linebreak[4]
$\pi_m(Y;Z/p^{a+b})$.
Consider the diagram:
\[
\xymatrix{
S^m
\ar@{->}[r]^{p^a}
&S^m
\ar@{->}[drr]^{\xi^{\prime\prime}}
&&\\
P^m(p^{a+b})
\ar@{->}[u]
\ar@{->}[r]
&P^m(p^b)
\ar@{->}[u]
\ar@{->}[r]
&P^m(p^{a+b})
\ar@{->}[r]^{\xi}
&
Y\\
&S^{m-1}
\ar@{->}[r]^{p^a}
\ar@{->}[u]
&S^{m-1}
\ar@{->}[u]
\ar@{->}[ur]_{\xi^{\prime}}&
.}
\]
Let $\xi^{\prime}$ be the restriction of $\xi$ to $S^{m-1}$. Since
$p^{a+b}\xi^{\prime}=0$, $\xi^{\prime}$ has
finite order. By hypothesis $p^a\xi^{\prime\prime}=0$ so an
extension
$\xi^{\prime\prime}\colon S^m\to Y$ exists. By hypothesis
$p^a\xi^{\prime\prime}=0$ so the
middle
horizontal composition is null homotopic. This composition
is $p^a\xi$.
\end{proof}
\begin{definition}\label{def5.6}
A space $Y$ is $(n,r)$-subexponential if for each
$s\geqslant 1$
\[
p^{r+s-1}\pi_{np^s}(Y;Z/p^{r+s})=0.
\]
\end{definition}
With this definition in hand, we have:
\begin{Theorem}\label{theor5.7}
\textup{(a)}\hspace*{0.5em}If $Z$ is any space which is $(2n,r)$-subexponential,
then every map $f\!\colon P^{2n}(p^r)\to \Omega Z$ extends to a
map $\hat{f}\!\colon T_{2n-1}(p^r)\to Z$.

\textup{(b)}\hspace*{0.5em}If $Z$ is a $(2n,r)$-subexponential H-space and there is an
H-map
\[
\hat{f}\!\colon T_{2n-1}(p^r)\to Z
\]
extending $f\!\colon P^{2n}(p^r)\to
Z$, then $\hat{f}$ is unique
up to homotopy.
\end{Theorem}
\begin{proof}
Part \textup{(a)} follows directly from~\cite[4.7]{AG}. Part~\textup{(b)} follows
from \ref{theor5.4} and \ref{lem5.5}.
\end{proof}

\section{}\label{sec6}
In this section we will discuss some of the possible options
for a universality property for $T_{2n-1}(p^r)$. We begin with the
observation that $ST_{2n-1}(p^r)\not\simeq P^{2n+1}(p^r)$. In fact, there is a
space
$G_{2n}(p^r)$ which is a retract of $ST_{2n-1}(p^r)$ and
$P^{2n+1}(p^r)\subset G_{2n}(p^r)\subset ST_{2n-1}(p^r)$.
$G_{2n}(p^r)$ contains cells of dimension $2np^i$ and $2np^i+1$ for each
$i\geqslant 0$ and no
other cells in positive dimensions (\cite[4.4c]{GT}). In particular, there
is no map $\phi\colon T_{2n-1}(p^r)\to \Omega P^{2n+1}(p^r)$ which induces
an
isomorphism in $\pi_{2n-1}$. In terms of~\ref{theor5.7}\textup{(a)} this is
correlated to
the fact that (\cite{CMN}):
\[
p^r\pi_{2np-1}(\Omega P^{2n+1}(p^r);Z/p^{r+1})\neq 0.
\]
We might ask for an H-map:
\[
T_{2n-1}(p^r)\to \Omega^2P^{2n+2}(p^r)
\]
inducing an isomorphism in $\pi_{2n-1}$. There clearly is a map since
$SG_{2n}(p^r)$ splits as a one point union of Moore spaces
(\cite[4.5]{GT}).
\begin{proposition}\label{prop6.1}
There is a unique H-map $\hat{f}\!\colon T_{2n-1}(p^t)\to
\Omega^2P^{2n+2}(p^r)$
extending the inclusion of $P^{2n}(p^r)$ iff the suspension map:
\[
T_{2n-1}(p^r)\xrightarrow{E}\Omega T_{2n}(p^r)
\]
is an H-map.
\end{proposition}
\begin{proof}
If $E$ is an H-map, we compose $E$ with the loops on the
map 
\[
g\colon T_{2n}(p^r)\to \Omega P^{2n+2}(p^r)
\]
(example 4c) to obtain
\[
T_{2n-1}(p^r)\to \Omega T_{2n}(p^r)\xrightarrow{\Omega g}\Omega^2
P^{2n+2}(p^r).
\]
Conversely, since $T_{2n}(p^r)$ is a retract of $\Omega P^{2n+2}(p^r)$ we
see that
we can construct an H-map
\[
T_{2n-1}(p^r)\xrightarrow{\ }\Omega^2P^{2n+2}(p^r)\xrightarrow{\Omega
h}\Omega T_{2n}(p^r).
\]
In fact, there is only one possible H-map $T_{2n-1}(p^r)\to \Omega
T_{2n}(p^r)$
and only one possible H-map $T_{2n-1}(p^r)\to \Omega^2P^{2n+2}(p^r)$
extending respectively the inclusions of $P^{2n}(p^r)$ by  \ref{theor5.7}\textup{(b)}.
\end{proof}
\begin{proposition}\label{prop6.2}
Suppose $T_{2n-1}(p^r)$ is homotopy associative. Then
the suspension map
\[
E\colon T_{2n-1}(p^r)\to \Omega T_{2n}(p^r)
\]
is an H-map.
\end{proposition}

Note: In \cite[1.2]{T1}, the author states that $T_{2n-1}(p^r)$ is
homotopy associative for $p\geqslant 5$. Unfortunately, the argument
given relies on some of the same assertions that
appear in the universality argument and more details are
needed to justify this claim.
\begin{proof}
The inclusion $G_{2n}\to ST_{2n-1}$ yields the homotopy
commutative diagram of fibrations (as in \ref{eq3.1})
\[
\xymatrix{
\Omega ST_{2n-1}
\ar@{->}[r]^{\partial}
\ar@{->}[d]
&T_{2n-1}
\ar@{-}[d]
\ar@{->}[r]
&T_{2n-1}\ast T_{2n-1}
\ar@{->}[r]
&ST_{2n-1}\\
\Omega G_{2n}
\ar@{->}[r]^h
&T_{2n-1}
\ar@{->}[r]
\ar@{->}[u]
&R
\ar@{->}[r]
\ar@{->}[u]
&G_{2n}.
\ar@{->}[u]
}
\]
(See also \cite[4.4]{GT}.) Since $T_{2n-1}$ is homotopy associative,
\[
\partial\colon \Omega ST_{2n-1}\to T_{2n-1}
\]
is an H-map; now $h\colon \Omega G_{2n}\to T_{2n-1}$ is the composition:
\[
\Omega G_{2n}\xrightarrow{\Omega f}\Omega ST_{2n-1}\xrightarrow{\partial}T_{2n-1}
\]
so $h$ is also an H-map. Also from \cite[4.4]{GT} we have a homotopy
commutative diagram of fibration sequences:
\[
\xymatrix{
\Omega G_{2n}
\ar@{->}[r]^h
\ar@{=}[d]
&T_{2n-1}
\ar@{->}[r]
\ar@{->}[d]^E
&R
\ar@{->}[r]
\ar@{->}[d]
&G_{2n}&\\
\Omega G_m
\ar@{->}[r]^{\Omega\varphi}
&\Omega T_{2n}
\ar@{->}[r]
&E
\ar@{->}[r]
&G_m
\ar@{=}[u]
\ar@{->}[r]^{\varphi}
&T_{2n}.
}
\]
\end{proof}
From this we see that $Eh\sim \Omega\varphi$ is an H-map. Now consider the
diagram:
\[
\xymatrix{
\Omega G_{2n}\times \Omega G_{2n}
\ar@{->}[r]^{h\times h}
\ar@{->}[d]
&T_{2n-1}\times T_{2n-1}
\ar@{->}[r]^{E\times E}
\ar@{->}[d]
&\Omega T_{2n}\times \Omega T_{2n}\ar@{->}[d]\\
\Omega G_{2n}
\ar@{->}[r]^{h}
&T_{2n-1}
\ar@{->}[r]^{E}
&\Omega T_{2n}.
}
\]
Since $Eh$ is an H-map, the rectangle commutes up to homotopy.
Since $h$ is an H-map, the left square does as well. But then
since $h\times h$ has a right homotopy universe, the right hand square
does and~$E$ is an H-map.
\begin{proposition}\label{prop6.3}
Suppose $Z$ is a double loop space and $Z$
is $(2n,r)$-\linebreak[4]exponential. Then each map $f\!\colon P^{2n}(p^r)\to Z$ has a unique
extension to an H-map $\hat{f}\!\colon
T_{2n-1}(p^r)\to Z$ iff $E\colon T_{2n-1}(p^r)\to \Omega T_{2n}(p^r)$
is an H map.
\end{proposition}
\begin{proof}
Suppose $E$ is an H map and $f\!\colon P^{2n}(p^r)\to Z$. Since $Z$ is
a double loop space, we can extend $f$ to an H map:
\[
\overline{f}\!\colon \Omega^2P^{2n+2}(p^r)\to Z.
\]
Since $T_{2n}(p^r)$ is a retract of $\Omega P^{2n+2}(p^r)$, we can form $\hat{f}$ as the
composition:
\[
T_{2n-1}(p^r)\xrightarrow{E}\Omega T_{2n}(p^r)\xrightarrow{\Omega i}\Omega^2P^{2n+2}(p^r)\xrightarrow{\overline{f}}Z.
\]
By \ref{theor5.4}, $\hat{f}$ is unique.

Conversely, if extensions all ways when the target
is an $(n,r)$-\linebreak[4]
subexponential double loop space, we can
construct an H map:
\[
T_{2n-1}(p^r)\to \Omega^2P^{2n+2}(p^r)
\]
since $p^r\pi_{2np-1}(\Omega^3P^{2n+2}(p^r);Z/p^{r+1})=0$. Now compose
this map with the loops on the map $\Omega P^{2n+2}(p^r)\to T_{2n}(p^r)$.
\end{proof}
\begin{corollary}\label{cor6.4}
If $T_{2n-1}(p^r)$ is homotopy associative and $Z$ is a double loop
space which is $(2n,r)$-subexponential, then each map $f\!\colon P^m(p^r)\to Z$
extends uniquely to an H-map $\hat{f}\!\colon T_{2n-1}(p^r)\to Z$.
\end{corollary}

This is a desirable universal property. However,
it is not sufficient for the program in~\cite{G4}. For that purpose
we would need to know that $T_m(p^r)$ is an acceptable
target for any~$m$.

We propose the following strengthening of~\ref{def5.6}.
\begin{definition}\label{def6.5}
A space $Z$ is strongly $(n,r)$-subexponential if
$p^{r+k}\pi_i(Z)\!{}=\nobreak{}\!0$ for all $i\leqslant np^{k+1}$ and $k\geqslant 0$.
\end{definition}
\begin{conjecture}\label{conj6.6}
$P^{2n}(p^r)\xrightarrow{i}T_{2n-1}(p^r)$ is universal for targets
which are strongly $(2n,r)$-subexponential Abelian
H~spaces.
\end{conjecture}
\begin{proposition}\label{prop6.7}
$T_{2n-1}(p^r)$ is strongly $(2n,r)$-subexponential.
\end{proposition}
\begin{proof}
Consider the fibration sequence:
\[
W_n\to T_{2n-1}(p^r)\to \Omega T_{2n}(p^r)
\]
from (\cite{G4}, \cite{AG}, \cite{GT}). By the results of Cohen, Moore and
Neisandorfer~\cite{CMN}, $p\pi_i(W_n)=0$. Neisendorfer
has proven that 
\[
p^{r}\pi_i(T_{2n}(p^r))=0
\]
(\cite{N1}), so we
conclude that
\[
p^{r+1}\pi_i(T_{2n-1}(p^r))=0.
\]
Thus it suffices to show that
\[
p^r\pi_i(T_{2n-1}(p^r))=0\quad\text{for}\ i\leqslant 2np.
\]
For $i\leqslant 2np$, $\pi_i(W_n)=0$ except when $i=2np-3$ and (in case $p=3$)
$i=2np$. We will examine these groups. According to (\cite{AG}, \cite{GT}),
$T_{2n-1}$ is a retract of $\Omega G_{2n}$. The $2np+1$ of~$G_{2n}$ is:
\[
P^{2n+1}(p^r)\cup CP^{2np}(p^{r+1}).
\]
Now $\pi_{2np-3}(T_{2n-1}(p^r))\subset \pi_{2np-2}(G_{2n})\cong \pi_{2np-2}(P^{2n+1}(p^r))$. This group
has order at most~$p^r$. Likewise $\pi_{2np}(T_{2n-1}(p^r))$ is contained
in $\pi_{2np+1}(G_{2n})$ which is a quotient of $\pi_{2np+1}(P^{2n+2}(p^r))$. This
group also has exponent~$p^r$.
\end{proof}

If conjecture~\ref{conj6.6} holds, we could show that
$p^r\cdot \pi_*(T_{2n-1}(p^r))=0$ and consequently each map $f\!\colon P^{m+1}(p^r)\to T_i$
has a unique extension to an H-map $\hat{f}\!\colon T_m(p^r)\to T_i$.
This is
sufficient for the purposes of~\cite{G4}.

It is interesting to note that this condition is precisely the
conclusion of a theorem of Barratt~\cite{Ba}.
\begin{Theorem}[Barratt]\label{theor6.8}
Suppose $SX$ is $n$-connected
and has characteristic~$p^r$. Then $SX$ is strongly $(n,r)$-subexponential.
\end{Theorem}

This invites the question of whether other $n$-connected
suspensions of characteristic~$p^r$ support a map
to an Abelian H-space which is universal for
targets that are strongly $(n,r)$-subexponential H-spaces.
Proposition~\ref{prop1.1} suggests that if so, these Abelian H-spaces
would have homology which is a symmetric algebra.

\nocite*{}
\bibliographystyle{amsalpha}
\bibliography{Gray-Abelian-H-spaces}

\end{document}